\documentclass[letter]{amsart}
\usepackage[utf8]{inputenc}
\usepackage{amsmath,amsfonts,amssymb,amsrefs,mathtools,amsthm}
\usepackage{enumerate}
\usepackage{graphicx}
\usepackage{multicol}
\usepackage[all]{xy}
\usepackage{tikz}
\usetikzlibrary{snakes,patterns}
\usetikzlibrary{decorations.pathreplacing,calligraphy}
\usetikzlibrary{decorations.pathmorphing}
\usepackage{caption}
\usepackage{cancel}

\usepackage{lipsum}

\usepackage[bookmarksnumbered, colorlinks, plainpages]{hyperref}
\hypersetup{colorlinks=true,linkcolor=red, anchorcolor=green, citecolor=cyan, urlcolor=blue, filecolor=magenta, pdftoolbar=true}
\usepackage[capitalise]{cleveref}

\newtheorem*{definition}{Definition}
\newtheorem{theorem}{Theorem}
\newtheorem{lemma}{Lemma}
\newtheorem{proposition}{Proposition}
\newtheorem{remark}{Remark}
\newtheorem{corollary}{Corollary}

\newtheorem{problem}{Problem}

\DeclareMathOperator{\dom}{dom}

\DeclareMathOperator{\lin_Q}{LIN_{\mathbb{Q}}}

\DeclareMathOperator{\covM}{\mathsf{cov}(\mathcal{M})}
\DeclareMathOperator{\covN}{\mathsf{cov}(\mathcal{N})}

\DeclareMathOperator{\cf}{\mathsf{cf}}
\DeclareMathOperator{\ZFC}{\mathsf{ZFC}}
\DeclareMathOperator{\CPA}{\mathsf{CPA}}
\DeclareMathOperator{\G}{\mathcal{G}}

\title[M set and SZ functions under rotations]{Mazurkiewicz Sets and Containment of Sierpi\'{n}ski-Zygmund Functions under Rotations}
\date{Draft of June 26, 2025}

\author{Cheng-Han Pan}
\address{Department of Mathematics and Statistics,\newline \indent Mount Holyoke College, South Hadley,\newline \indent Massachusetts 01075-1461, United States}
\email{\href{mailto:cpan@mtholyoke.edu}{cpan@mtholyoke.edu}}

\thanks{The author would like to thank professor K.C. Ciesielski and the anonymous referees for their careful reading of the manuscript, and, in particular, for their many insightful comments and valuable suggestions which helped to improve the clarity and quality of this work.}


\begin{document}
\begin{abstract}
A Mazurkiewicz set is a plane subset that intersect every straight line at exactly two points, and a Sierpi\'{n}ski-Zygmund function is a function from $\mathbb{R}$ into $\mathbb{R}$ that has as little of the standard continuity as possible. Building on the recent work of Kharazishvili, we construct a Mazurkiewicz set that contains a Sierpi\'{n}ski-Zygmund function in every direction and another one that contains none in any direction. Furthermore, we show that whether a Mazurkiewicz set can be expressed as a union of two Sierpi\'{n}ski-Zygmund functions is independent of Zermelo-Fraenkel set theory with the Axiom of Choice (ZFC). Some open problems related to the containment of Hamel functions are stated.
\end{abstract}
\maketitle

\section{Introduction}
A Mazurkiewicz set, also known as a two-point set, is a subset of $\mathbb{R}^2$ that intersects every straight line of in the plane at exactly two points. This straightforward definition leads to two important properties:
\begin{itemize}
\item[•] Mazurkiewicz sets are closed under rotations about the origin.
\item[•] A Mazurkiewicz set is a union of two real functions from $\mathbb{R}$ into $\mathbb{R}$.
\end{itemize}

Recently, Kharazishvili has shown that {\em there is a Mazurkiewicz set containing the graph of a Sierpi\'{n}ski-Zygmund function} in \cite{AK}. The insight from this result prompted our investigation into the containment of Sierpi\'{n}ski-Zygmund functions in a Mazurkiewicz set. In \cref{S3}, we extend Kharazishvili's result to consider rotations, demonstrating that {\em there is a Mazurkiewicz set containing the graph of a Sierpi\'{n}ski-Zygmund function in every direction}. Conversely, we also show that {\em there is a Mazurkiewicz set containing no graph of any Sierpi\'{n}ski-Zygmund function in any direction}. As a Mazurkiewicz set is always a union of two real functions, we wonder {\em whether a Mazurkiewicz set can be a union of two Sierpi\'{n}ski-Zygmund functions in every direction}, and we find, in \cref{S4}, that the answer cannot be decided within the Zermelo-Fraenkel set theory with the axiom of choice included, denoted as $\ZFC$. Lastly in \cref{S5}, we look into the relation between Mazurkiewicz sets and Hamel functions in sense of P{\l}otka and state some open problems for future investigation.

In this article, we work within the framework of subsets of $\mathbb{R}^2$ and often identify functions and partial functions from $\mathbb{R}$ into $\mathbb{R}$ with their graphs, which are also subsets of $\mathbb{R}^2$. In particular, the three primary classes of subsets of $\mathbb{R}^2$ we focus on are {\em Mazurkiewicz sets}, {\em Sierpi\'{n}ski-Zygmund set}, and {\em Hamel basis} of $\mathbb{R}^2$. To acquaint readers with the terminologies, we provide their definitions, properties and some history in \cref{S2}.

\section{Preliminaries}\label{S2}
We begin with a brief overview of the {\em Mazurkiewicz set}, which was first introduced in \cite{SM}.

\begin{definition}
A set $M\subseteq\mathbb{R}^2$ is called a Mazurkiewicz set provided that it intersects every straight line of $\mathbb{R}^2$ in a doubleton.
\end{definition}

Despite its straightforward definition, the properties of Mazurkiewicz sets are quite intriguing. For example, a Mazurkiewicz set can be, but not necessarily, Lebesgue measurable.\footnote{See \cite[Section 21, Chapter 10]{MR1996162}, \cite[Chapter 6]{MR2067444}, or \cite{MR3674890}.} Although It has been shown in \cites{MR1800242,MR0315589} that no Mazurkiewicz set can be a $F_\sigma$ set, whether a Mazurkiewicz set can be a Borel set or not is still unknown.\footnote{This long-standing unsolved problem is stated and discussed by Mauldin in \cite{MR1078668}. However, Mauldin mentioned in \cite{MR1620829} that ``I believe I first heard it from Erd\"{o}s, who said it had been around since he was a baby''.} Other interesting properties include but not limited to that a Mazurkiewicz set must not contain an arc,\footnote{One may conclude this result from \cites{MR1955662,MR1093599}.} and that for all $2\leq\kappa,\lambda<\mathfrak{c}$, there exists a $\kappa$-point set which is homeomorphic to a $\lambda$-point set.\footnote{See \cite{MR2784793}. Moreover, its author also showed that there is a two-point set that is homeomorphic to a $\lambda$-point set for every $2\leq\lambda<\mathfrak{c}$ under the set-theoretical assumptions of $\mathfrak{c}=\cf(\mathfrak{c})$ and $\mathfrak{c}<\aleph_\mathfrak{c}$} Since our results also consider the rotations of Mazurkiewicz sets, what interests us the most is the result in \cite{MR2421831} that there exists a Mazurkiewicz set which is invariant under $\mathfrak{c}$ many rotations about the origin.

To construct a Mazurkiewicz set, it is important to avoid having three points on the same line. Suppose $A\subseteq\mathbb{R}^2$ has at most two points on the same line in the middle of the construction process. Any new point to be added needs to avoid those straight lines that already meet $A$ in doubletons. For convenience, we will always let
\begin{itemize}
\item[$\mathcal{L}(A)$] denote the collection of all straight lines of $\mathbb{R}^2$ passing through any pair of two distinct points in $A\subseteq\mathbb{R}^2$.
\end{itemize}
It is easy to see that $|\mathcal{L}(A)|<\mathfrak{c}$ whenever $|A|<\mathfrak{c}$.\footnote{\label{FNLA<c}$|\mathcal{L}(A)|\leq|A|\otimes|A|=|A|<\mathfrak{c}$. Of course, we omit the dull case of $A$ being finite.} This simple fact makes the transfinite inductions in the proof of \cref{PMS,LMNSZR,TMHFHF} possible. The next proposition provides a well-known construction a Mazurkiewicz set that will be used in \cref{TMSZSZR}.

\begin{proposition}\label{PMS}
If $B\subseteq\mathbb{R}^2$ intersects every straight line of $\mathbb{R}^2$ in a set of cardinality $\mathfrak{c}$, then $B$ contains some Mazurkiewicz sets.
\end{proposition}
\begin{proof}
Let $\{\ell_\xi\}_{\xi<\mathfrak{c}}$ be an injective enumeration of all straight lines of $\mathbb{R}^2$. By transfinite induction on $\xi<\mathfrak{c}$, we construct a sequence $\{E_\xi\}_{\xi<\mathfrak{c}}$ of disjoint subsets of $B$ whose union $M\coloneqq\bigcup_{\xi<\mathfrak{c}}E_\xi$ is a Mazurkiewicz set. For every $\xi<\mathfrak{c}$, simply pick a subset $E_\xi\subseteq(\ell_\xi\cap B)\setminus\bigcup\mathcal{L}(\bigcup_{\zeta<\xi}E_\zeta)$ such that $|\ell_\xi\cap\bigcup_{\zeta\leq\xi}E_\xi|=2$. Note that there is enough room for $E_\xi$ since $|\ell_\xi\cap B|=\mathfrak{c}$ and $|\mathcal{L}(\bigcup_{\zeta\leq\xi}E_\zeta)|<\mathfrak{c}$.
\end{proof}

What is a {\em Sierpi\'{n}ski-Zygmund set}?

\begin{definition}
A set $X\subseteq\mathbb{R}^2$ is called a Sierpi\'{n}ski-Zygmund set provided that it intersects each continuous partial function in a set of cardinality less than $\mathfrak{c}$.
\end{definition}

The notion of Sierpi\'{n}ski-Zygmund sets was introduced by P{\l}otka in \cite[Definition 18]{MR1911699} as a natural generalization of the function $f\colon\mathbb{R}\to\mathbb{R}$ constructed by Sierpi\'{n}ski and Zygmund in \cite{SZ},\footnote{``Le but de cette note est de d\'{e}duire du th\'{e}or\`{e}me de Zermelo l'existence d'une fonction d'une variable r\'{e}elle $f(x)$ qui est discontinue sur tout ensemble de puissance du continu.''} where $f$ is discontinuous whenever its domain is restricted to a subset of cardinality $\mathfrak{c}$. We now refer to any function that meets this criterion as a {\em Sierpi\'{n}ski-Zygmund function}.\footnote{It is worth mentioning that this classical result of Sierpi\'{n}ski and Zygmund was essentially motivated by a theorem of Blumberg in \cite[Theorem III]{MR1501216}, stating that, for any function $g\colon\mathbb{R}\to\mathbb{R}$, there exists an everywhere dense subset $D$ of $\mathbb{R}$ such that the restriction $g\restriction_D$ is continuous on $D$.

According to Blumberg's construction, the set $D$, being dense in $\mathbb{R}$, is infinite and countable. The existence of a Sierpi\'{n}ski-Zygmund function $f\colon\mathbb{R}\to\mathbb{R}$ shows that one cannot assert, within $\ZFC$, the uncountability of $D$. Consequently, if the continuum hypothesis holds, then the restriction of the same function to any uncountable subset of $\mathbb{R}$ is discontinuous. Moreover, the Sierpi\'{n}ski-Zygmund construction yields an example of a function that is simultaneously nonmeasurable in the Lebesgue sense, by Lusin's theorem, and does not possess the Baire property as shown in \cite[Theorem 8.38]{MR1321597}. For more information, see \cite{MR3999051}, a comprehensive survey of Sierpi\'{n}ski-Zygmund functions for more information.} When we identify functions and partial functions with their graphs, we can equivalently define that a Sierpi\'{n}ski-Zygmund function is a function that intersects each continuous partial function in a set of cardinality less than $\mathfrak{c}$. Obviously, every Sierpi\'{n}ski-Zygmund function qualifies as a Sierpi\'{n}ski-Zygmund set. On the other hand, if a subset of a Sierpi\'{n}ski-Zygmund set happens to be the graph of a function, then it is automatically a Sierpi\'{n}ski-Zygmund function. Moreover, a union of less than $\cf(\mathfrak{c})$-many Sierpi\'{n}ski-Zygmund sets is still a Sierpi\'{n}ski-Zygmund set.

The next proposition provides a {\em standard} method of determining whether a set is Sierpi\'{n}ski-Zygmund or not. By a {\em standard} method, we refer to a foundational approach that underpins the construction the construction of any Sierpi\'{n}ski-Zygmund set or function, including the first one in \cite{SZ}. For convenience, we will always let
\begin{itemize}
\item[$\mathcal{G}$] denote the family of all continuous functions defined on a $G_\delta$ subset of $\mathbb{R}$.
\end{itemize}

\begin{proposition}\label{PSZ}
For $X\subseteq\mathbb{R}^2$, the following statements are equivalent.
\begin{enumerate}[(i)]
\item $X$ is a Sierpi\'{n}ski-Zygmund set.
\item $|X\cap g|<\mathfrak{c}$ for every $g\in\mathcal{G}$.
\end{enumerate}
\end{proposition}
\begin{proof}
(i) implies (ii) is trivial. To see (ii) implies (i), suppose that $X$ is not a Sierpi\'{n}ski-Zygmund set. Let $f\colon\mathbb{R}\to\mathbb{R}$ be a continuous partial function such that $|f\cap X|=\mathfrak{c}$. By a theorem of Kuratowski,\footnote{See \cite[Theorem 1, Section 35]{MR0217751} or \cite[Theorem 3.8]{MR1321597}.} $f$ has a continuous extension $g\supseteq f$ such that $\dom(g)$ is a $G_\delta$ subset of $\mathbb{R}$ which contradicts (ii) as $|g\cap X|\geq|f\cap X|=\mathfrak{c}$ and $g\in\mathcal{G}$.
\end{proof}

Note that $|\mathcal{G}|=\mathfrak{c}$.\footnote{Since there are $\mathfrak{c}$ many $G_\delta$ subsets in $\mathbb{R}$ and each continuous function is uniquely determined by the values on a countable dense set of its domain, we are able to conclude $\mathfrak{c}\leq|\mathcal{G}|\leq\mathfrak{c}\otimes(\omega\otimes\mathfrak{c})=\mathfrak{c}$.} If $\{g_\xi\}_{\xi<\mathfrak{c}}$ is an injective enumeration of $\mathcal{G}$, the next proposition shows another method, which is employed in \cref{TMSZR}, to identify whether a function is Sierpi\'{n}ski-Zygmund or not.

\begin{proposition}\label{PSZD}
Let $D\coloneqq\mathbb{R}^2\setminus\bigcup_{\zeta\leq<\xi<\mathfrak{c}}\{\langle x_\xi,g_\zeta(x_\xi)\rangle\}$, where $\{g_\xi\}_{\xi<\mathfrak{c}}$ and $\{x_\xi\}_{\xi<\mathfrak{c}}$ are injective enumerations of $\mathcal{G}$ and $\mathbb{R}$ respectively. If $f\colon\mathbb{R}\to\mathbb{R}$ is a function such that $|f\setminus D|<\mathfrak{c}$, then $f$ is a Sierpi\'{n}ski-Zygmund function.
\end{proposition}
\begin{proof}
Such an $f$ is a Sierpi\'{n}ski-Zygmund since it satisfies condition (ii) of \cref{PSZ}. Fix any $g_\xi\in\mathcal{G}$. Note that $|f\cap g_\xi|=|(f\cap D)\cap g_\xi|\oplus|(f\setminus D)\cap g_\xi|$. It is easy to see that $|D\cap g_\xi|<\mathfrak{c}$ and therefore $|(f\cap D)\cap g_\xi|\leq|D\cap g_\xi|<\mathfrak{c}$. On the other hand, $|(f\setminus D)\cap g_\xi|\leq|f\setminus D|<\mathfrak{c}$.
\end{proof}

Now we provide the definition and an introduction of a {\em Hamel basis}.

\begin{definition}
A set $H\subseteq\mathbb{R}^2$ is called a Hamel basis provided that it is a basis of $\mathbb{R}^2$ as a linear space over $\mathbb{Q}$.
\end{definition}

The story of a Hamel basis begins with Cauchy's additive functional equation. It was shown by Cauchy in \cite[Chapitre V]{CE} that the continuous solutions to the equation $f(x+y)=f(x)+f(y)$ must be linear.\footnote{``1\textsuperscript{er}. Probl\`{e}me. D\'{e}terminer la fonction $f(x)$, de mani\`{e}re qu'elle reste continue entre deux limites r\'{e}elles quelconques de la variable $x$, et que l'on ait pour toutes les valeurs r\'{e}elles des variables $x$ et $y$ $\phi(x+y)=\phi(x)+\phi(y)$.''} On the other hand, Hamel constructed the first discontinuous solution,\footnote{In \cite[Section 5.2]{MR2467621}, the author stated that ``Discontinuous additive functions are sometimes called Hamel functions. They exhibit many pathological properties, as will be seen later in this book''. On the other hand, the notion of a Hamel function is defined differently by P{\l}otka in \cite[Section 3]{MR1948092}, and our discussion in \cref{S5} involves the later one to which we refer as Hamel functions in sense of P{\l}otka to avoid confusion.} and his construction in \cite{MR1511317} is based on a subset of $\mathbb{R}$ such that each $x\in\mathbb{R}$ can be uniquely expressed as a linear combination of the subset over $\mathbb{Q}$.\footnote{``1) Es existiert eine Busis aller Zahlen, d. h. es gibt eine Menge von Zahlen $a,b,c,\cdots$ derart, da{\ss} sich jede Zahl $x$ in einer und auch nur einer Weise in der Form $x=\alpha a+\beta b+\gamma c+\cdots$ darstellen l\"{a}{\ss}t, wo die Zahlen $\alpha,\beta,\gamma,\cdots$ rational sind, aber in jedem einzelnen Falle nur eine endliche Anzahl von ihnen von Null verschieden ist.''} The subset is indeed a basis of $\mathbb{R}$ as a linear space over $\mathbb{Q}$, which is now called the {\em Hamel basis}. Certainly, this result can be generalized to higher-dimensional $\mathbb{R}^n$'s, and we are particularly interested in the two-dimensional case. The next remark shows a simple relation between Hamel bases of $\mathbb{R}^2$ and Mazurkiewicz sets.

\begin{remark}\:\label{R3MH}
\begin{enumerate}[(i)]
\item There is a Hamel basis that is also a Mazurkiewicz set.
\item There is a Hamel basis that is not a Mazurkiewicz set.
\item There is a Mazurkiewicz set that is not a Hamel basis.
\end{enumerate}
\end{remark}
\begin{proof}
(i) has been shown in \cite[Theorem 2.2]{MR3165512}. To see (ii), let $X\subseteq\mathbb{R}$ be a Hamel basis of $\mathbb{R}$. $(X\times\{0\})\cup(\{0\}\times X)$ is a Hamel basis of $\mathbb{R}^2$ that is clearly not a Mazurkiewicz set. To see (iii), any Mazurkiewicz set containing the origin cannot be a Hamel basis.
\end{proof}

\section{Two Examples of Mazurkiewicz Sets with Opposite Properties in Containment of a Sierpi\'{n}ski-Zygmund Function in every direction}\label{S3}
Recall that it was shown in \cite{AK} that there is a Mazurkiewicz containing the graph of a Sierpi\'{n}ski-Zygmund function. In this section, we show that there is a Mazurkiewicz set containing the graph of a Sierpi\'{n}ski-Zygmund function in every direction. On the opposite, we also show that there is a Mazurkiewicz set containing no graph of a Sierpi\'{n}ski-Zygmund function in any direction. To clarify our discussions on directions and rotations, we will include supplementary figures.

We first introduce two convenient notations that help describe the rotations in \cref{LSZR,LMNSZR}. If $\ell\subseteq\mathbb{R}^2$ is a straight line and $A\subseteq\mathbb{R}^2$, then we let
\begin{itemize}
\item[$\mathcal{L}_{\parallel\ell}$] denotes the collection of all straight lines parallel to $\ell$, and
\end{itemize}
\begin{itemize}
\item[{$\rho_\ell[A]$}] denotes the rotation of $A$ about the origin, in positive direction by an angle in $[0,\pi)$, such that $\ell$ becomes a vertical line to $\rho_\ell[A]$ as illustrated in \cref{F1}.
\begin{figure}[h]
\begin{tikzpicture}
\draw[white!0](0,0)--(3,1.5)node[pos=1,anchor=west,black]{$A$};
\draw[rotate around={90-66.037511:(6.5,0)},white!0](0+6.5,0)--(3+6.5,1.5)node[rotate={90-66.037511},pos=1,anchor=west,black]{$\rho_\ell[A]$};

\draw[->](-1.5,0)--(4,0)node[anchor=west]{$x$};
\draw[->](0,-1.5)--(0,3)node[anchor=west]{$y$};
\draw[very thick](1,3)--(3,-1.5)node[pos=1,anchor=south west]{$\ell$};
\draw[ultra thick,decorate,decoration={random steps,segment length=3pt,amplitude=1pt},dashed](-1,-0.5)rectangle(3,1.5);

\draw[ultra thick,->](4,1.5)--(5,1.5)node[pos=1/2,anchor=south]{$\rho_\ell$};

\draw[dashed,gray,ultra thin,->,rotate around={90-66.037511:(6.5,0)}](-1.5+6.5,0)--(4+6.5,0)node[anchor=west]{$x$};
\draw[dashed,gray,ultra thin,->,rotate around={90-66.037511:(6.5,0)}](0+6.5,-1.5)--(0+6.5,3)node[anchor=west]{$y$};
\draw[->](-1.5+6.5,0)--(4+6.5,0)node[anchor=west]{$x$};
\draw[->](0+6.5,-1.5)--(0+6.5,3)node[anchor=west]{$y$};
\draw[very thick](1+6.5,3)--(3+6.5,-1.5)node[pos=1,anchor=south west]{$\ell$};
\draw[ultra thick,rotate around={90-66.037511:(6.5,0)},decorate,decoration={random steps,segment length=3pt,amplitude=1pt},dashed](-1+6.5,-0.5)rectangle(3+6.5,1.5);
\end{tikzpicture}
\caption{An illustration that demonstrates how the rotation $\rho_\ell$ acts on a set $A\subseteq\mathbb{R}^2$ with respect to a straight line $\ell\subseteq\mathbb{R}^2$ such that $\ell$ becomes a vertical line to $\rho_\ell[A]$.}\label{F1}
\end{figure}
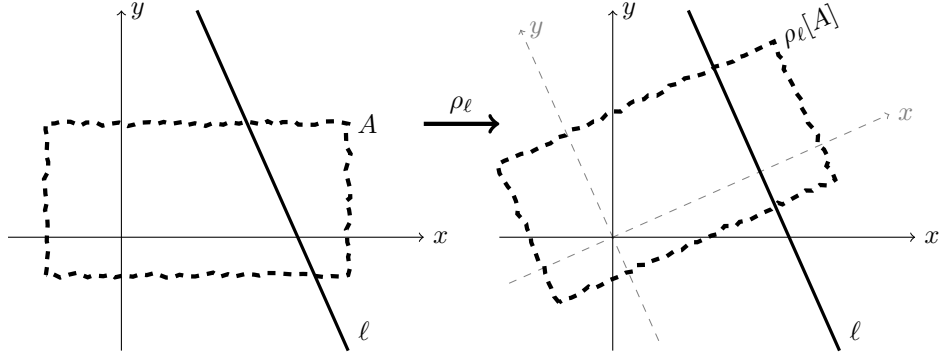
\end{itemize}

\begin{lemma}\label{LSZR}
If $D\subseteq\mathbb{R}^2$ intersects every vertical line of $\mathbb{R}^2$ in a set of cardinality $\mathfrak{c}$, then there is a Mazurkiewicz set $M$ such that for every straight line $\ell\subseteq\mathbb{R}^2$, the set
\begin{equation*}
Z_\ell\coloneqq\{\hat{\ell}\in\mathcal{L}_{\parallel\ell}\colon M\cap\hat{\ell}\cap\rho_\ell[D]=\emptyset\}
\end{equation*}
has cardinality at most $1$.
\end{lemma}
\begin{proof}
Let $\{\ell_\xi\}_{\xi<\mathfrak{c}}$ be an injective enumeration of all straight lines of $\mathbb{R}^2$. By transfinite induction on $\xi<\mathfrak{c}$, we construct a sequence $\{E_\xi\}_{\xi<\mathfrak{c}}$ of mutually disjoint subsets of $\mathbb{R}^2$ whose union $M\coloneqq\bigcup_{\xi<\mathfrak{c}}E_\xi$ is a desired Mazurkiewicz set. For every $\xi<\mathfrak{c}$, we choose $E_\xi$ such that
\begin{itemize}
\item[(I)] $|E_\xi|\leq2$. Additionally, $E_\xi\cap\ell_\xi\cap\rho_{\ell_\xi}[D]\neq\emptyset$ whenever $|\ell_\xi\cap\bigcup_{\zeta<\xi}E_\zeta|<2$.
\item[(P)] $|\ell\cap\bigcup_{\zeta\leq\xi}E_\zeta|\leq2$ for every straight line $\ell\subseteq\mathbb{R}^2$. Additionally, $|\ell\cap\bigcup_{\zeta\leq\xi}E_\zeta|<2$ whenever $\ell\in\{\ell_\zeta\}_{\zeta>\xi}$ is parallel to some line in $\mathcal{L}(\bigcup_{\zeta\leq\xi}E_\zeta)$ and $|\ell\cap\bigcup_{\zeta<\xi}E_\zeta|<2$.
\item[(D)] $|\ell_\xi\cap\bigcup_{\zeta\leq\xi}E_\zeta|=2$.
\end{itemize}

Obviously, $M\coloneqq\bigcup_{\xi<\mathfrak{c}}E_\xi$ is already a Mazurkiewicz set without considering the additional conditions in the construction. To see that $|Z_\ell|\leq1$ for every straight line $\ell\subseteq\mathbb{R}^2$, we fix an $\ell\subseteq\mathbb{R}^2$ and let $\kappa<\mathfrak{c}$ be the smallest ordinal number such that $\ell_\kappa$ is parallel to $\ell$, that is, the first $\ell_\kappa\in\mathcal{L}_{\parallel\ell}$. We want to show that $\ell_\kappa$ is the only line in $\mathcal{L}_{\parallel\ell}$ such that $M\cap\ell_\kappa\cap\rho_{\ell}[D]=\emptyset$ could happen.\footnote{Note that we cannot require $|Z_\ell|=0$ because $M\cap\ell_\kappa\cap\rho_\ell[D]=\emptyset$ may happen. In particular, if $|(\ell_\kappa\setminus\rho_{\ell}[D])\cap\bigcup_{\zeta<\kappa}E_\zeta|=2$ occurs, then we are not able to pick an nonempty $E_\kappa\subseteq\ell_\kappa\cap\rho_\ell[D]$ to ensure $E_\kappa\cap\ell_\kappa\cap\rho_\ell[D]\neq\emptyset$ without violating (D).} By way of contradiction, we suppose that there is some $\ell_\xi\in\mathcal{L}_{\parallel\ell}\setminus\{\ell_\kappa\}$ such that $M\cap\ell_\xi\cap\rho_{\ell}[D]=\emptyset$. Then $|\ell_\xi\cap\bigcup_{\zeta<\xi}E_\zeta|=2$ must be the case, otherwise $M\cap\ell_\xi\cap\rho_\ell[D]\supseteq E_\xi\cap\ell_\xi\cap\rho_\ell[D]\neq\emptyset$ by the additional condition in (I). Let $\alpha<\mathfrak{c}$ be the smallest ordinal number such that $|\ell_\xi\cap\bigcup_{\zeta\leq\alpha}E_\zeta|=2$. As $\ell_\xi$ is parallel to $\ell_\kappa\in\mathcal{L}(\bigcup_{\zeta\leq\kappa}E_\zeta)\subseteq\mathcal{L}(\bigcup_{\zeta\leq\alpha}E_\zeta)$ and $|\ell_\xi\cap\bigcup_{\zeta<\alpha}E_\zeta|<2$ by the definition of $\alpha$, (P) ensures $|\ell_\xi\cap\bigcup_{\zeta\leq\alpha}E_\zeta|<2$, which is a contradiction.

Now we describe our choice of $E_\xi$ while $\{E_\zeta\}_{\zeta<\xi}$ satisfies (I), (P), (D). Let $\mathcal{P}(\bigcup_{\zeta<\xi}E_\zeta)$ denote the collection of all straight lines parallel to some line in $\mathcal{L}(\bigcup_{\zeta<\xi}E_\zeta)$ and passing through some point in $\bigcup_{\zeta<\xi}E_\zeta$. Note that $|\mathcal{P}(\bigcup_{\zeta<\xi}E_\zeta)|<\mathfrak{c}$ for every $\xi<\mathfrak{c}$.\footnote{$|\mathcal{P}(\bigcup_{\zeta<\xi}E_\zeta)|\leq|\bigcup_{\zeta<\xi}E_\zeta|\otimes|\mathcal{L}(\bigcup_{\zeta<\xi}E_\zeta)|\leq|\bigcup_{\zeta<\xi}E_\zeta|\otimes|\bigcup_{\zeta<\xi}E_\zeta|=|\bigcup_{\zeta<\xi}E_\zeta|<\mathfrak{c}$. Of course, we omit the dull case of $\bigcup_{\zeta<\xi}E_\zeta$ being finite.} Since $|\ell_\xi\cap\rho_{\ell_\xi}[D]|=\mathfrak{c}$ and $|\mathcal{P}(\bigcup_{\zeta<\xi}E_\zeta)\setminus\{\ell_\xi\}|<\mathfrak{c}$, we can pick
\begin{center}
$E_\xi\subseteq(\ell_\xi\cap\rho_{\ell_\xi}[D])\setminus\underbrace{\bigcup\textstyle(\mathcal{P}(\bigcup_{\zeta<\xi}E_\zeta)\setminus\{\ell_\xi\})}_{(P)}$
\end{center}
such that $|\ell_\xi\cap\bigcup_{\zeta\leq\xi}E_\zeta|=2$. By definition, it is easy to see that $\mathcal{L}(\bigcup_{\zeta<\xi}E_\zeta)\subseteq\mathcal{P}(\bigcup_{\zeta<\xi}E_\zeta)$, and therefore (D) holds,\footnote{Compared with the construction in \cref{PMS}, $\bigcup_{\xi<\mathfrak{c}}E_\xi$ already suffices to be a Mazurkiewicz set if we require $E_\xi\subseteq(\ell_\xi\cap\rho_{\ell_\xi}[D])\setminus\bigcup\mathcal{L}(\bigcup_{\zeta<\xi}E_\zeta)$. Here we increase $\mathcal{L}(\bigcup_{\zeta<\xi}E_\zeta)$ to $\mathcal{P}(\bigcup_{\zeta<\xi}E_\zeta)$, so this construction preserves more properties. Note that we need to remove $\ell_\xi$ from $\mathcal{P}(\bigcup_{\zeta<\xi}E_\zeta)$ under the consideration of the case that $\ell_\xi\in\mathcal{P}(\bigcup_{\zeta<\xi}E_\zeta)$ but $|\ell_\xi\cap\bigcup_{\zeta<\xi}E_\zeta|=1$.} that is $\bigcup_{\xi<\mathfrak{c}}E_\xi$ is a Mazurkiewicz set. To see that our definition of $\mathcal{P}(\bigcup_{\zeta<\xi}E_\zeta)\setminus\{\ell_\xi\}$ guarantees the additional condition in (P), fix an $\ell\in\{\ell_\zeta\}_{\zeta>\xi}$ parallel to some line in $\mathcal{L}(\bigcup_{\zeta\leq\xi}E_\zeta)$ and $|\ell\cap\bigcup_{\zeta<\xi}E_\zeta|<2$. We want to ensure that $E_\xi$ is chosen in a way that $|\ell\cap\bigcup_{\zeta\leq\zeta}E_\xi|<2$. If $\ell\in\mathcal{P}(\bigcup_{\zeta<\xi}E_\xi)\setminus\{\ell_\xi\}$, then clearly $|\ell\cap\bigcup_{\zeta\leq\zeta}E_\xi|<2$ as $E_\xi\cap\ell=\emptyset$. If $\ell\not\in\mathcal{P}(\bigcup_{\zeta<\xi}E_\xi)\setminus\{\ell_\xi\}$, then $|\ell\cap\bigcup_{\zeta<\xi}E_\zeta|=0$ and $|E_\xi\cap\ell|\leq|\ell_\xi\cap\ell|\leq1$ together implies that $|\ell\cap\bigcup_{\zeta\leq\zeta}E_\xi|<2$.
\end{proof}

\begin{theorem}\label{TMSZR}
There is a Mazurkiewicz set containing the graph of a Sierpi\'{n}ski-Zygmund function in every direction.
\end{theorem}
\begin{proof}
Let $D\coloneqq\mathbb{R}^2\setminus\bigcup_{\zeta\leq<\xi<\mathfrak{c}}\{\langle x_\xi,g_\zeta(x_\xi)\rangle\}$, where $\{g_\xi\}_{\xi<\mathfrak{c}}$ and $\{x_\xi\}_{\xi<\mathfrak{c}}$ are injective enumerations of $\mathcal{G}$ and $\mathbb{R}$ respectively. It is clear that set $D$ satisfies the premise of $\cref{LSZR}$, and therefore we have a Mazurkiewicz set $M\subseteq\mathbb{R}^2$ such that
\begin{equation*}
Z_\ell\coloneqq\{\hat{\ell}\in\mathcal{L}_{\parallel\ell}\colon M\cap\hat{\ell}\cap\rho_\ell[D]=\emptyset\}
\end{equation*}
has cardinality at most $1$ for every straight line $\ell\subseteq\mathbb{R}^2$. Fix any straight line $\ell$ and rotate $M$, $\ell$, $\rho_\ell[D]$ backward about the origin with $\rho_\ell^{-1}$ such that $\ell$ becomes a vertical line as illustrated in \cref{F2}.
\begin{figure}[h]
\begin{tikzpicture}
\draw[thin,dashed,gray!50](1+0.8,3)--(3+0.8-4/18,-1.5+0.5);
\draw[thin,dashed,gray!50](1+0.6,3)--(3+0.6-4/18,-1.5+0.5);
\draw[thin,dashed,gray!50](1+0.4,3)--(3+0.4-4/18,-1.5+0.5);
\draw[thin,dashed,gray!50](1+0.2,3)--(3+0.2-4/18,-1.5+0.5);
\draw[thin,dashed,gray!50](1-0.2,3)--(3-0.2,-1.5);
\draw[thin,dashed,gray!50](1-0.4,3)--(3-0.4,-1.5);
\draw[thin,dashed,gray!50](1-0.6,3)--(3-0.6,-1.5);
\draw[thin,dashed,gray!50](1-0.8,3)--(3-0.8,-1.5);
\draw[thin,dashed,gray!50](1-1,3)--(3-1,-1.5);
\draw[thin,dashed,gray!50](189/97+6.5+0.8,3)--(189/97+6.5+0.8,-1);
\draw[thin,dashed,gray!50](189/97+6.5+0.6,3)--(189/97+6.5+0.6,-1);
\draw[thin,dashed,gray!50](189/97+6.5+0.4,3)--(189/97+6.5+0.4,-1);
\draw[thin,dashed,gray!50](189/97+6.5+0.2,3)--(189/97+6.5+0.2,-1);
\draw[thin,dashed,gray!50](189/97+6.5-0.2,3)--(189/97+6.5-0.2,-1.5);
\draw[thin,dashed,gray!50](189/97+6.5-0.4,3)--(189/97+6.5-0.4,-1.5);
\draw[thin,dashed,gray!50](189/97+6.5-0.6,3)--(189/97+6.5-0.6,-1.5);
\draw[thin,dashed,gray!50](189/97+6.5-0.8,3)--(189/97+6.5-0.8,-1.5);
\draw[thin,dashed,gray!50](189/97+6.5-1,3)--(189/97+6.5-1,-1.5);

\draw[white!0](0,0)--({2.5*cos(120)+1},{1.4*sin(120)+1})node[pos=1,anchor=south east,black]{$\bigcup_{\zeta<\xi}E_\zeta$};
\draw[rotate={90-66.037511},white!0](0,0)--(3,1.5)node[rotate={90-66.037511},pos=1,anchor=west,black]{$\rho_\ell[D]$};
\draw[rotate around={66.037511-90:(6.5,0)},white!0](0+6.5,0)--({2.5*cos(150)+1+6.5},{1.4*sin(150)+1})node[rotate={66.037511-90},pos=1,anchor=south east,black]{$\rho_\ell^{-1}[M]$};
\draw[white!0](0+6.5,0)--(3+6.5,1.5)node[pos=1,anchor=west,black]{$\rho_\ell^{-1}[\rho_\ell[D]]$};

\draw[ultra thick,->](4,1.5)--(5,1.5)node[pos=1/2,anchor=south]{$\rho_\ell^{-1}$};

\draw[->](-1.5,0)--(4,0)node[anchor=west]{$x$};
\draw[->](0,-1.5)--(0,3)node[anchor=west]{$y$};
\draw[very thick](1,3)--(3,-1.5)node[pos=1,anchor=south west]{$\ell$};
\draw[thick,rotate={90-66.037511},decorate,decoration={random steps,segment length=3pt,amplitude=1pt},dashed](-1,-0.5)rectangle(3,1.5);
\draw[ultra thick,decorate,decoration={random steps,segment length=3pt,amplitude=1pt},dashed](1,1)ellipse(2.5 and 1.4);

\draw[->](-1.5+6.5,0)--(4+6.5,0)node[anchor=west]{$x$};
\draw[->](0+6.5,-1.5)--(0+6.5,3)node[anchor=west]{$y$};
\draw[very thick](189/97+6.5,3)--(189/97+6.5,-1.5)node[rotate={66.037511-90},pos=1,anchor=south west]{$\rho_\ell^{-1}[\ell]$};
\draw[thick,decorate,decoration={random steps,segment length=3pt,amplitude=1pt},dashed](-1+6.5,-0.5)rectangle(3+6.5,1.5);
\draw[ultra thick,rotate around={66.037511-90:(6.5,0)},decorate,decoration={random steps,segment length=3pt,amplitude=1pt},dashed](1+6.5,1)ellipse(2.5 and 1.4);
\end{tikzpicture}
\caption{An illustration that demonstrates how the backward rotation $\rho_\ell^{-1}$ about the origin acts on $M$, $\ell$, $\rho_\ell[D]$ in the proof of \cref{TMSZR}.}\label{F2}
\end{figure}
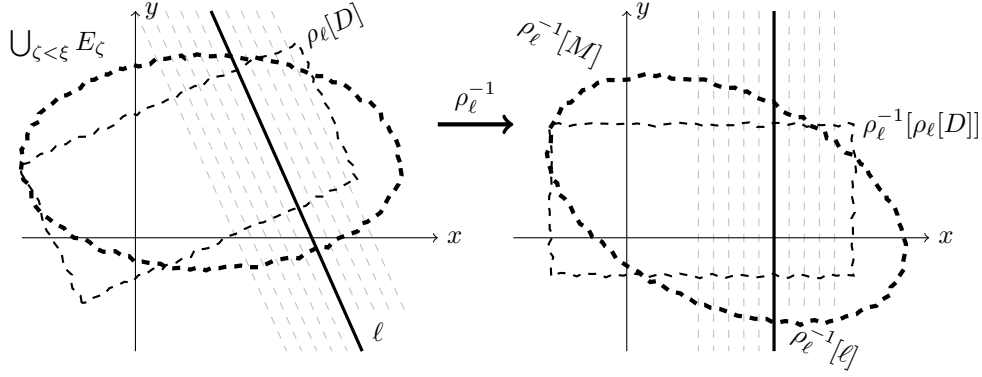
Note that $\rho_\ell^{-1}[\rho_\ell[D]]=D$. Since $\rho_\ell^{-1}[M]$ and $D$ have nonempty intersections in all but at most one vertical line, we may pick a function $f\colon\mathbb{R}\to\mathbb{R}$ such that $f\subseteq\rho_\ell^{-1}[M]$ and $|f\setminus D|\leq1$. According to \cref{PSZD}, such an $f$ is a Sierpi\'{n}ski-Zygmund function contained in $\rho_\ell^{-1}[M]$.
\end{proof}

\begin{lemma}\label{LMNSZR}
If $D\subseteq\mathbb{R}^2$ intersects every straight line of $\mathbb{R}^2$ in a set of cardinality at most $\omega$, and additionally that the set 
\begin{equation*}
Z_\ell\coloneqq\{\hat{\ell}\in\mathcal{L}_{\parallel\ell}\colon|\hat{\ell}\cap D|\geq2\}
\end{equation*}
has cardinality $\mathfrak{c}$ for every straight line $\ell$, then there exists a Mazurkiewicz set $M$ such that the set $\{\hat{\ell}\in\mathcal{L}_{\parallel\ell}\colon|M\cap\hat{\ell}\cap D|=2\}$ has cardinality $\mathfrak{c}$ for every straight line $\ell$.
\end{lemma}
\begin{proof}
Let $\{\ell_\xi\}_{\xi<\mathfrak{c}}$ be an injective enumeration of all straight lines of $\mathbb{R}^2$. By transfinite induction on $\xi<\mathfrak{c}$, we construct a sequence $\{E_\xi\}_{\xi<\mathfrak{c}}$ of mutually disjoint subsets of $\mathbb{R}^2$ whose union $M\coloneqq\bigcup_{\xi<\mathfrak{c}}E_\xi$ is a desired Mazurkiewicz set. For every $\xi<\mathfrak{c}$, we choose $E_\xi$ such that
\begin{itemize}
\item[(I)] $|E_\xi|\leq4$. Additionally, there is an $\ell_\xi^*\in Z_{\ell_\xi}$ such that $\ell_\xi^*\cap\bigcup_{\zeta<\xi}E_\zeta=\emptyset$ and $|E_\xi\cap\ell_\xi^*\cap D|=2$.
\item[(P)] $|\ell\cap\bigcup_{\zeta\leq\xi}E_\zeta|\leq2$ for every straight line $\ell\subseteq\mathbb{R}^2$.
\item[(D)] $|\ell_\xi\cap\bigcup_{\zeta\leq\xi}E_\zeta|=2$.
\end{itemize}
We first pick $E_\xi'\subseteq\ell_\xi\setminus\bigcup\mathcal{L}(\bigcup_{\zeta<\xi}E_\xi)$ such that $|\ell_\xi\cap(E_\xi'\cup\bigcup_{\zeta<\xi}E_\xi)|=2$. Let $M_\xi'\coloneqq E_\xi'\cup\bigcup_{\zeta<\xi}E_\xi$. Clearly, $M_\xi'$ has cardinality strictly less than $\mathfrak{c}$, and so do $\mathcal{L}(M_\xi')$ and $M_\xi''\coloneqq D\cap\bigcup\mathcal{L}(M_\xi')$.\footnote{Recall \cref{FNLA<c} and note that $|M_\xi''|=|D\cap\bigcup\mathcal{L}(M_\xi')|\leq|\bigcup_{\ell\in\mathcal{L}(M_\xi')}D\cap\ell|\leq|\mathcal{L}(M_\xi')|\otimes\omega<\mathfrak{c}$.} Since $|Z_{\ell_\xi}|=\mathfrak{c}$, there is an $\ell_\xi^*\in Z_{\ell_\xi}$ disjoint from $M_\xi'\cup M_\xi''$. Note that $|\ell_\xi^*\cap D|\geq2$, and therefore we can pick a doubleton $E_\xi''\subseteq\ell_\xi^*\cap D$ and let $E_\xi\coloneqq E_\xi'\cup E_\xi''$.
\end{proof}

\begin{theorem}\label{TMNSZR}
There is a Mazurkiewicz set containing no graph of a Sierpi\'{n}ski-Zygmund function in every direction.
\end{theorem}
\begin{proof}
Let $D$ be the unit circle in $\mathbb{R}^2$. Note that it satisfies the premise of \cref{LMNSZR}, and the Mazurkiewicz set $M\subseteq\mathbb{R}$ constructed there is as needed. Fix any straight line $\ell$ and rotate $M$, $\ell$, $D$ backward about the origin with $\rho_\ell^{-1}$ such that $\rho_\ell^{-1}[\ell]$ becomes a vertical line as illustrate in \cref{F3}.
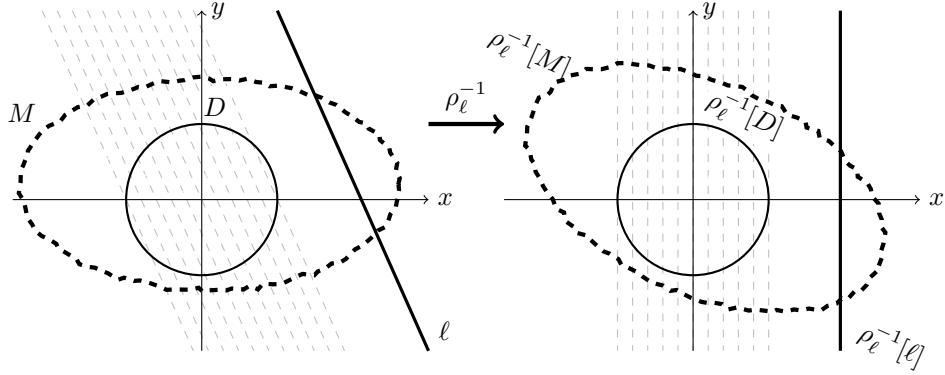
\begin{figure}[h]
\begin{tikzpicture}
\draw[thin,dashed,gray!50](-10/9+1,2.5)--(8/9+1,-2);
\draw[thin,dashed,gray!50](-10/9+0.8,2.5)--(8/9+0.8,-2);
\draw[thin,dashed,gray!50](-10/9+0.6,2.5)--(8/9+0.6,-2);
\draw[thin,dashed,gray!50](-10/9+0.4,2.5)--(8/9+0.4,-2);
\draw[thin,dashed,gray!50](-10/9+0.2,2.5)--(8/9+0.2,-2);
\draw[thin,dashed,gray!50](-10/9,2.5)--(8/9,-2);
\draw[thin,dashed,gray!50](-10/9-0.2,2.5)--(8/9-0.2,-2);
\draw[thin,dashed,gray!50](-10/9-0.4,2.5)--(8/9-0.4,-2);
\draw[thin,dashed,gray!50](-10/9-0.6,2.5)--(8/9-0.6,-2);
\draw[thin,dashed,gray!50](-10/9-0.8,2.5)--(8/9-0.8,-2);
\draw[thin,dashed,gray!50](-10/9-1,2.5)--(8/9-1,-2);
\draw[thin,dashed,gray!50](0+6.5+1,2.5)--(0+6.5+1,-2);
\draw[thin,dashed,gray!50](0+6.5+0.8,2.5)--(0+6.5+0.8,-2);
\draw[thin,dashed,gray!50](0+6.5+0.6,2.5)--(0+6.5+0.6,-2);
\draw[thin,dashed,gray!50](0+6.5+0.4,2.5)--(0+6.5+0.4,-2);
\draw[thin,dashed,gray!50](0+6.5+0.2,2.5)--(0+6.5+0.2,-2);
\draw[thin,dashed,gray!50](0+6.5,2.5)--(0+6.5,-2);
\draw[thin,dashed,gray!50](0+6.5-0.2,2.5)--(0+6.5-0.2,-2);
\draw[thin,dashed,gray!50](0+6.5-0.4,2.5)--(0+6.5-0.4,-2);
\draw[thin,dashed,gray!50](0+6.5-0.6,2.5)--(0+6.5-0.6,-2);
\draw[thin,dashed,gray!50](0+6.5-0.8,2.5)--(0+6.5-0.8,-2);
\draw[thin,dashed,gray!50](0+6.5-1,2.5)--(0+6.5-1,-2);
\fill[white](0,1)rectangle(0.4,1.5);
\fill[white,rotate around={66.037511-90:(6.5,0)}](-0.4+6.5,1)rectangle(0.8+6.5,1.6);

\draw[white!0](0,0)--({2.5*cos(150)+0.1},{1.4*sin(150)+0.2})node[pos=1,anchor=south east,black]{$M$};
\draw[white!0](0,0)--({cos(80)},{sin(80)})node[pos=1,anchor=south,black]{$D$};
\draw[rotate around={66.037511-90:(6.5,0)},white!0](0+6.5,0)--({2.5*cos(150)+0.1+6.5},{1.4*sin(150)+0.2})node[rotate={66.037511-90},pos=1,anchor=east,black]{$\rho_\ell^{-1}[M]$};
\draw[rotate around={66.037511-90:(6.5,0)},white!0](0+6.5,0)--({cos(80)+6.5},{sin(80)})node[rotate={66.037511-90},pos=1,anchor=south,black]{$\rho_\ell^{-1}[D]$};

\draw[->](-2.5,0)--(3,0)node[anchor=west]{$x$};
\draw[->](0,-2)--(0,2.5)node[anchor=west]{$y$};
\draw[very thick](1,2.5)--(3,-2)node[pos=1,anchor=south west]{$\ell$};
\draw[thick,rotate={90-66.037511}](0,0)circle(1);
\draw[ultra thick,decorate,decoration={random steps,segment length=3pt,amplitude=1pt},dashed](0.1,0.2)ellipse(2.5 and 1.4);

\draw[ultra thick,->](3,1)--(4,1)node[pos=1/2,anchor=south]{$\rho_\ell^{-1}$};

\draw[->](-2.5+6.5,0)--(3+6.5,0)node[anchor=west]{$x$};
\draw[->](0+6.5,-2)--(0+6.5,2.5)node[anchor=west]{$y$};
\draw[very thick](189/97+6.5,2.5)--(189/97+6.5,-2)node[rotate={66.037511-90},pos=1,anchor=south west]{$\rho_\ell^{-1}[\ell]$};
\draw[thick](0+6.5,0)circle(1);
\draw[ultra thick,rotate around={66.037511-90:(6.5,0)},decorate,decoration={random steps,segment length=3pt,amplitude=1pt},dashed](0.1+6.5,0.2)ellipse(2.5 and 1.4);
\end{tikzpicture}
\caption{An illustration that demonstrates how the backward rotation $\rho_\ell^{-1}$ about the origin acts on $M$, $\ell$, $D$ in the proof of \cref{TMNSZR}.}\label{F3}
\end{figure}
Note that $\rho_\ell^{-1}[D]=D$. Let $f$ be any function contained in $\rho_\ell^{-1}[M]$. As there are $\mathfrak{c}$-many vertical lines intersecting $\rho_\ell^{-1}[M]\cap\rho_\ell^{-1}[D]$ in a doubleton, $f$ must meet at least one of the upper unit semicircle and the lower unit semicircle $\mathfrak{c}$-many times. Since both semicircles are continuous partial functions, $f$ cannot be a Sierpi\'{n}ski-Zygmund function.
\end{proof}

So far, we have a good understanding of constructing a Mazurkiewicz set containing a Sierpi\'{n}ski-Zygmund function, or conversely, one that does not. Recall that a Mazurkiewicz set is always a union of two real functions from $\mathbb{R}$ into $\mathbb{R}$. The next natural question is whether a Mazurkiewicz set can also be a union of two Sierpi\'{n}ski-Zygmund functions. In the next section, we will show that this question cannot be decided within $\ZFC$.

\section{Mazurkiewicz set as a union of two Sierpi\'{n}ski-Zygmund functions in every direction}\label{S4}
This section has two parts. In the first part, we show that it is consistent with $\ZFC$ that no Mazurkiewicz set can be a union of two Sierpi\'{n}ski-Zygmund functions under the set-theoretical assumption of the {\em covering property axiom}, abbreviated as the $\CPA$.\footnote{Note that the $\CPA$ holds in the Sacks model and is consistent with $\ZFC$. The exposition of the $\CPA$ can be found in the monograph \cite{MR2176267}.} In the second part, we will construct a Mazurkiewicz set that is a union of two Sierpi\'{n}ski-Zygmund functions in every direction under another set-theoretical assumption that $\mathbb{R}$ cannot be covered by less than $\mathfrak{c}$-many meager sets,\footnote{We say that $M\subseteq\mathbb{R}$ is meager provided that $M$ is a countable union of nowhere dense subsets of $\mathbb{R}$.} denoted by $\covM=\mathfrak{c}$. We know that $\covM=\mathfrak{c}$ is also consistent with $\ZFC$ as it follows from the {\em continuum hypothesis}.

\begin{lemma}\label{LMBB}
Every Mazurkiewicz set is a union of two bijections on $\mathbb{R}$.
\end{lemma}
\begin{proof}
Let $M\subseteq\mathbb{R}^2$ be a Mazurkiewicz set. We define a graph $G_M$ induced by $M$ in such a way that the $M$ is the set of vertexes and there is an edge between two distinct points in $M$ whenever they share the same vertical or horizontal line of $\mathbb{R}^2$. Evidently, $G_M$ is $2$-regular and a component of $G_M$ is either a cycle or a double ray.\footnote{While $G$ is a $2$-regular connected graph, our terminologies are consistent with definitions in \cite[Section 8.1]{MR3822066}. In particualr, If $G$ has finitely many vertexes, we call $G$ a cycle. If $G$ can be indexed by $\mathbb{Z}$, where adjacent vertexes are indexed consecutively, we call $G$ a double ray.} If we show that each component is bipartite or equivalently $2$-colorable, then each component is a union of two injective partial functions from $\mathbb{R}$ into $\mathbb{R}$. Moreover, partial functions from different components must have disjoint domains and disjoint images. In the case of a cycle, the cycle intersects finitely many vertical lines and each of which contains exactly two distinct vertexes from the cycle. Since the cycle has an even order, the cycle is bipartite. In the case of a double ray, it is clearly bipartite since its vertexes can be partitioned into the even-indexed and the odd-indexed. \cref{F4} demonstrates an example of a double ray in $G_M$ with no two vertexes of the same partition share the same edge which is either a vertical or horizontal line segment.

\begin{figure}[h]
\begin{tikzpicture}
\tikzstyle{vb}=[fill,circle,,minimum size=2pt,inner sep=2pt]
\tikzstyle{vw}=[draw,fill=white,circle,minimum size=2pt,inner sep=2pt]
\draw[thin,dashed](-1,-0.5)--(4,-0.5);
\draw[thin,dashed](-1,0)--(4,0);
\draw[thin,dashed](-1,0.7)--(4,0.7);
\draw[thin,dashed](-1,1.5)--(4,1.5);
\draw[thin,dashed](-1,2)--(4,2);
\draw[thin,dashed](-1,2.5)--(4,2.5);
\draw[thin,dashed](-1,3)--(4,3);
\draw[thin,dashed](-0.5,-1)--(-0.5,3.5);
\draw[thin,dashed](0.5,-1)--(0.5,3.5);
\draw[thin,dashed](1.5,-1)--(1.5,3.5);
\draw[thin,dashed](2.2,-1)--(2.2,3.5);
\draw[thin,dashed](3.2,-1)--(3.2,3.5);
\draw[thin,dashed](3.5,-1)--(3.5,3.5);
\node[](-7)at(-2,0){$\cdots$};
\node[vb](-6)at(3.2,0){};
\node[vw](-5)at(3.2,1.5){};
\node[vb](-4)at(-0.5,1.5){};
\node[vw](-3)at(-0.5,-0.5){};
\node[anchor=north east]at(-0.5,-0.5){$\textbf{p}_{-3}$};
\node[vb](-2)at(2.2,-0.5){};
\node[anchor=north west]at(2.2,-0.5){$\textbf{p}_{-2}$};
\node[vw](-1)at(2.2,0.7){};
\node[anchor=south west]at(2.2,0.7){$\textbf{p}_{-1}$};
\node[vb](0)at(0.5,0.7){};
\node[anchor=north east]at(0.5,0.7){$\textbf{p}_0$};
\node[vw](1)at(0.5,2.5){};
\node[anchor=south east]at(0.5,2.5){$\textbf{p}_1$};
\node[vb](2)at(1.5,2.5){};
\node[anchor=north west]at(1.5,2.5){$\textbf{p}_2$};
\node[vw](3)at(1.5,3){};
\node[anchor=south east]at(1.5,3){$\textbf{p}_3$};
\node[vb](4)at(3.5,3){};
\node[vw](5)at(3.5,2){};
\node[](6)at(5,2){$\cdots$};
\draw[->,very thick](-7)--(-6);
\draw[->,very thick](-6)--(-5);
\draw[->,very thick](-5)--(-4);
\draw[->,very thick](-4)--(-3);
\draw[->,very thick](-3)--(-2);
\draw[->,very thick](-2)--(-1);
\draw[->,very thick](-1)--(0);
\draw[->,very thick](0)--(1);
\draw[->,very thick](1)--(2);
\draw[->,very thick](2)--(3);
\draw[->,very thick](3)--(4);
\draw[->,very thick](4)--(5);
\draw[->,very thick](5)--(6);
\end{tikzpicture}
\caption{An illustration that demonstrates how a double ray $\{\ldots,\textbf{p}_{-2},\textbf{p}_{-1},\textbf{p}_0,\textbf{p}_1,\textbf{p}_2,\ldots\}\subseteq G_M$ described in the proof of \cref{LMBB} is partitioned into the black vertexes and the white vertexes such that no edge contain vertexes from the same partition.}\label{F4}
\end{figure}
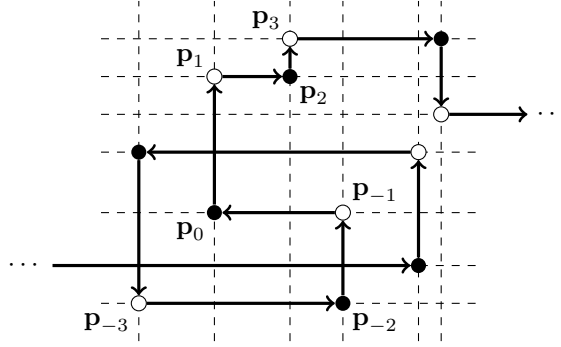

Finally, take one injective partial function from each of the components. Note that their union is a bijective function, and the compliment of the bijective function in $M$ is the other bijective function.
\end{proof}

\begin{theorem}\label{TMNSZSZ}
Assume the $\CPA$, which is consistent with $\ZFC$. There is no Sierpi\'{n}ski-Zygmund Mazurkiewicz set.
\end{theorem}
\begin{proof}
Suppose there is a Sierpi\'{n}ski-Zygmund Mazurkiewicz set $M\subseteq\mathbb{R}^2$. By \cref{LMBB}, there are two bijective functions $f,g\colon\mathbb{R}\to\mathbb{R}$ such that $f\cup g=M$. Since $f$ is a subset of $M$, it is automatically a Sierpi\'{n}ski-Zygmund function. However, this is a contradiction because the image of a Sierpi\'{n}ski-Zygmund function cannot contain any perfect set under the assumption of the $\CPA$ by \cite[Theorem 4.3]{MR3999051}.
\end{proof}

\begin{corollary}\label{CMNSZSZ}
Assume the $\CPA$, which is consistent with $\ZFC$. No Mazurkiewicz set is a union of two Sierpi\'{n}ski-Zygmund functions.
\end{corollary}
\begin{proof}
Suppose $f,g\colon\mathbb{R}\to\mathbb{R}$ are two Sierpi\'{n}ski-Zygmund functions such that $f\cup g$ is a Mazurkiewicz set. But $f\cup g$ is also a Sierpi\'{n}ski-Zygmund set which is impossible under the assumption of the $\CPA$ by \cref{TMNSZSZ}.
\end{proof}

In the next part, we will construct a Sierpi\'{n}ski-Zygmund Mazurkiewicz set under the assumption of $\covM=\mathfrak{c}$. In addition, we want the Sierpi\'{n}ski-Zygmund Mazurkiewicz set to be closed under rotations about the origin. In order to verify that the set we construct is Sierpi\'{n}ski-Zygmund in every direction, we let
\begin{itemize}
\item[$\mathcal{G}_r$] denote the collection of the graphs of each continuous function\\
from a $G_\delta$ subset of $\mathbb{R}$ into $\mathbb{R}$ {\em and their rotations about the origin.}
\end{itemize}
It is easy to see that $|\G_r|=\mathfrak{c}$ and that $X\subseteq\mathbb{R}^2$ is a Sierpi\'{n}ski-Zygmund set in every direction if and only if $|X\cap g|<\mathfrak{c}$ for every $g\in\G_r$ by \cref{PSZ}. Indeed, $\mathcal{G}_r$ is the core of our construction and will be involved in the following \cref{LL,TMSZSZR}.

\begin{lemma}\label{LL}
Fix any $g\in\G_r$ and let $\mathcal{L}$ be the collection of all straight lines $\ell\cong\mathbb{R}$ of $\mathbb{R}^2$ in which $g\cap\ell$ is somewhere dense in $\ell$,\footnote{We say that $D\subseteq\mathbb{R}$ is somewhere dense in $\mathbb{R}$ provided that there exists a nonempty open interval $J\subseteq\mathbb{R}$ in which $D\cap J$ is dense.} then $|\mathcal{L}|\leq\omega$.
\end{lemma}
\begin{proof}
Note that when $g\in\G_r$ is the graph of a rotated function, we identify its domain as a $G_\delta$ subset of some line $\ell_g\cong\mathbb{R}$ passing through the origin, and indeed $\dom(g)$ is a projection that projects the graph of $g$ into $\ell_g$. For every straight line $\ell\in\mathcal{L}$, there exists an open interval $J_\ell\subseteq\ell\cong\mathbb{R}$ in which $g\cap J_\ell$ is dense. Simultaneously, there is also an open interval $I_\ell\subseteq\ell_g$ in which $\dom(g\cap J_\ell)$ is dense. Note that a continuous function is uniquely determined by its values on a dense set.

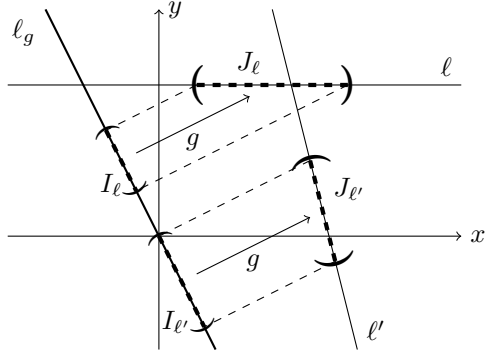
\begin{figure}[h]
\begin{tikzpicture}
\draw[->](-2,0)--(4,0)node[anchor=west]{$x$};
\draw[->](0,-1.5)--(0,3)node[anchor=west]{$y$};
\draw[thick](-1.5,3)--(3/4,-1.5)node[pos=0,anchor=north east]{$\ell_g$};
\draw[](-2,2)--(4,2)node[pos=1,anchor=south east]{$\ell$};
\draw[ultra thick,dashed](0.5,2)--(2.5,2)node[rotate=0,scale=1.5,pos=0]{$\textbf{(}$}node[rotate=0,scale=1,pos=1/2,anchor=south east]{$J_{\ell}$}node[rotate=0,scale=1.5,pos=1]{$\textbf{)}$};
\draw[thin,dashed](0.5,2)--(-7/10,7/5);
\draw[thin,dashed](2.5,2)--(-3/10,3/5);
\draw[ultra thick,dashed](-7/10,7/5)--(-3/10,3/5)node[rotate=-63.44,scale=1,pos=0]{$\textbf{(}$}node[rotate=0,scale=1,pos=0.9,anchor=east]{$I_{\ell}$}node[rotate=-63.44,scale=1,pos=1]{$\textbf{)}$};
\draw[domain=-0.3:1.2,->]plot({\x},{0.5*\x+1.25});
\node[anchor=north]at(0.45,1.475){$g$};
\draw[](1.5,3)--(2.625,-1.5)node[pos=1,anchor=south west]{$\ell'$};
\draw[ultra thick,dashed](2,1)--(7/3,-1/3)node[rotate=-75.96,scale=1.5,pos=0]{$\textbf{(}$}node[rotate=0,scale=1,pos=1/2,anchor=south west]{$J_{\ell'}$}node[rotate=-75.96,scale=1.5,pos=1]{$\textbf{)}$};;
\draw[thin,dashed](2,1)--(0,0);
\draw[thin,dashed](7/3,-1/3)--(0.6,-1.2);
\draw[ultra thick,dashed](0,0)--(0.6,-1.2)node[rotate=-63.44,scale=1,pos=0]{$\textbf{(}$}node[rotate=0,scale=1,pos=0.95,anchor=east]{$I_{\ell'}$}node[rotate=-63.44,scale=1,pos=1]{$\textbf{)}$};
\draw[domain=0.5:2,->]plot({\x},{0.5*\x-0.75});
\node[anchor=north]at(1.25,-0.125){$g$};
\end{tikzpicture}
\caption{An illustration that demonstrates the key ideas in the proof of \cref{LL}, where the thick dashed line segments $J_\ell,J_{\ell'}$ are open intervals in which $g\cap\ell,g\cap\ell'$ are dense respectively, and $I_\ell,I_{\ell'}$ are also open intervals in which $\dom(g\cap J_\ell),\dom(g\cap J_{\ell'})$ are dense respectively.}\label{F5}
\end{figure}

As illustrated in \cref{F5}, it is easy to see that $\ell\neq\ell'$ implies $J_\ell\cap J_{\ell'}=\emptyset$ and therefore $I_\ell\cap I_{\ell'}=\emptyset$. Since $\{I_\ell\}_{\ell\in\mathcal{L}}$ is a collection of mutually disjoint open intervals of $\ell_g$, we conclude that $|\mathcal{L}|\leq\omega$.
\end{proof}

\begin{theorem}\label{TMSZSZR}
Assume $\covM=\mathfrak{c}$, which is consistent with $\ZFC$. There is a Sierpi\'{n}ski-Zygmund Mazurkiewicz set that is closed under rotations about the origin.
\end{theorem}
\begin{proof}
Let $\{g_\xi\}_{\xi<\mathfrak{c}}$ and $\{\ell_\xi\}_{\xi<\mathfrak{c}}$ be injective enumerations respectively for $\mathcal{G}_r$ and all straight lines of $\mathbb{R}^2$. For every $\xi<\mathfrak{c}$, we consider $g_\xi^*\coloneqq g_\xi\setminus\bigcup\mathcal{L}_\xi$, where $\mathcal{L}_\xi$ is the collection of all straight lines $\ell\cong\mathbb{R}$ of $\mathbb{R}^2$ in which $g_\xi\cap\ell$ is somewhere dense. Note that $g_\xi^*\cap\ell$ is nowhere dense in every $\ell$, and therefore each $\ell_\xi^*\coloneqq\ell_\xi\setminus\bigcup_{\zeta<\xi}g_\zeta^*$ has cardinality $\mathfrak{c}$ due to the assumption of $\covM=\mathfrak{c}$.\footnote{Suppose $|\ell_\xi\setminus\bigcup_{\zeta<\xi}g_\zeta^*|<\mathfrak{c}$. Note that each singleton in $\ell_\xi\setminus\bigcup_{\zeta<\xi}g_{\zeta}^*$ is also a meager set. This results a contradiction to $\covM=\mathfrak{c}$ as $\mathbb{R}\cong\ell_\xi=(\ell_\xi\setminus\bigcup_{\zeta<\xi}g_\zeta^*)\cup\bigcup_{\zeta<\xi}g_\zeta^*$ suggests that $\mathbb{R}$ can be covered by $|\ell_\xi\setminus\bigcup_{\zeta<\xi}g_\zeta^*|\oplus\xi$-many meager sets and $|\ell_\xi\setminus\bigcup_{\zeta<\xi}g_\zeta^*|\oplus\xi<\mathfrak{c}=\covM$.} Evidently, $\bigcup_{\xi<\mathfrak{c}}\ell_\xi^*$ contains some Mazurkiewicz set by \cref{PMS}.

To see that any Mazurkiewicz set $M\subseteq\bigcup_{\xi<\mathfrak{c}}\ell_\xi^*$ is also a Sierpi\'{n}ski-Zygmund set in every direction, we fix any $g_\eta\in\mathcal{G}_r$ and show $|g_\eta\cap M|<\mathfrak{c}$. Since $g_\eta\subseteq g_\eta^*\cup\bigcup\mathcal{L}_\eta$ and $|M\cap\bigcup\mathcal{L}_\eta|\leq2\otimes|\mathcal{L}_\eta|\leq\omega$ by \cref{LL}, we only need to show that $|M\cap g_\eta^*|<\mathfrak{c}$. Notice that our construction of $\ell_\xi^*\coloneqq\ell_\xi\setminus\bigcup_{\zeta<\xi}g_\zeta^*$ implies that $g_\eta^*\cap\ell_\xi^*=\emptyset$ whenever $\xi>\eta$. As $M\subseteq\bigcup_{\xi<\mathfrak{c}}\ell_\xi^*$, we have
\begin{center}
$g_\eta\cap M=g_\eta\cap M\cap\bigcup_{\xi<\mathfrak{c}}\ell_\xi^*=g_\eta\cap M\cap\bigcup_{\xi\leq\eta}\ell_\xi^*\subseteq M\cap\bigcup_{\xi\leq\eta}\ell_\xi^*$.
\end{center}
Lastly, $|g_\eta\cap M|\leq|M\cap\bigcup_{\xi\leq\eta}\ell_\xi^*|\leq2\otimes\eta<\mathfrak{c}$ finishes the proof.
\end{proof}

\begin{corollary}\label{CMSZSZR}
Assume $\covM=\mathfrak{c}$, which is consistent with $\ZFC$. There is a Mazurkiewicz set that is a union of two Sierpi\'{n}ski-Zygmund functions in every direction.
\end{corollary}
\begin{proof}
By \cref{TMSZSZR}, we have a Sierpi\'{n}ski-Zygmund Mazurkiewicz set that is closed under rotations about the origin. By \cref{LMBB}, we know that each of its rotation is a union of two bijections on $\mathbb{R}$. Note that each of which is a subset of a Sierpi\'{n}ski-Zygmund set $M$ and therefore each of which is also a Sierpi\'{n}ski-Zygmund function.
\end{proof}

\begin{remark}
The same results in \cref{TMSZSZR} and as well as in \cref{CMSZSZR} can be obtained by assuming $\mathbb{R}$ is not a union of less than $\mathfrak{c}$-many sets of Lebesgue measure zero in $\mathbb{R}$, denoted by $\covN=\mathfrak{c}$.\footnote{See \cite[Theorem 4.4]{MR3999051} for a similar construction of a Sierpi\'{n}ski-Zygmund with Darboux property under the assumption of $\covN=\mathfrak{c}$.} Note that $\covN=\mathfrak{c}$ is also consistent with $\ZFC$ and that there are at most countable many straight lines $\ell\cong\mathbb{R}$ in which $g\cap\ell$ has positive Lebesgue measure for every $g\in\mathcal{G}_r$.
\end{remark}

\section{Related open problems and remarks}\label{S5}
In the previous two sections, we explored whether a Mazurkiewicz set must contain one or two Sierpi\'{n}ski-Zygmund functions in every direction. One interesting property that a Hamel basis and a Mazurkiewicz set share is their closure under rotations about the origin. However, the graph of a rotated function is usually not a function anymore. This raises the question of whether a Mazurkiewicz set can contain a {\em Hamel functions} in sense of P{\l}otka in some or all directions.

\begin{definition}
A function $h$ from $\mathbb{R}$ into $\mathbb{R}$ is called a Hamel function in sense of P{\l}otka provided that its graph is a Hamel basis of $\mathbb{R}^2$.
\end{definition}

Since there is no confusion in the context, we will simply call it a {\em Hamel function} in the rest of our discussion. The definition and as well as the existence of such a function was introduced by P{\l}otka in \cite[Section 3]{MR1948092}. Although it is a relatively new class of functions, there is a sequence of studies on Hamel functions in \cites{MR2527126,MR2518967,MR3396995,MR2177444,MR2591764,MR3016414,MR3016415}.

The next remark describes a simple structure that prevents a Mazurkiewicz set from containing any Hamel function. Note that $\lin_Q(A)$ denotes the smallest smallest linear subspace of $\mathbb{R}^2$ over $\mathbb{Q}$ containing $A\subseteq\mathbb{R}^2$.\footnote{$\mathrm{span}_\mathbb{Q}(A)$ is another commonly used notation for $\lin_Q(A)$.}

\begin{remark}\label{RMNHF}
If a Mazurkiewicz set $M\subseteq\mathbb{R}^2$ contains a linearly independent subset $D\subseteq M$ such that $|D|<|V|$, where $V$ is the collection of all vertical lines $\ell$ of $\mathbb{R}^2$ such that $\ell\cap M\subseteq\lin_Q(D)$, then $M$ does not contain the graph of a Hamel function. Indeed, $D$ may be as simple as $\{\langle1,0\rangle,\langle0,1\rangle\}$, and any Mazurkiewicz set containing $\{\langle0,0\rangle,\langle0,1\rangle\}$, $\{\langle1,0\rangle,\langle1,2\rangle\}$, and $\{\langle3,1\rangle,\langle3,2\rangle\}$ is an example.
\end{remark}
\begin{proof}
Suppose $f\subseteq M$ is a Hamel function. We show that $f\cap\lin_Q(D)$ has a strictly larger cardinality than $D$, and therefore it cannot be linearly independent over $\mathbb{Q}$. Let $\ell$ be a vertical line. Note that if $\ell\cap M\subseteq\lin_Q(D)$, then $f\cap\ell\subseteq\lin_Q(D)$. Thus, $|D|<|V|\leq|f\cap\lin_Q(D)|$ is as needed.
\end{proof}

However, not much is known when rotations about the origin are considered, and therefore we state our first problem and its opposite version.

\begin{problem}\label{PMNHFR}
Is there a Mazurkiewicz set containing no graph of a Hamel function in every direction?
\end{problem}

\begin{problem}\label{PMHFR}
Is there a Mazurkiewicz set containing the graph of a Hamel function in every direction?
\end{problem}

If \cref{PMHFR} has a positive answer, then we may further ask our third problem.

\begin{problem}\label{PMHFHFR}
Is there a Mazurkiewicz set that is a union of two Hamel functions in every direction?
\end{problem}

Without rotations, it is not difficult to construct a Mazurkiewicz set that is a union of two Hamel functions. In order to construct such an example, we prepare our notations and a few simple facts. Let $A\subseteq\mathbb{R}$, our next construction will involve $\lin_Q(A)$, $\mathcal{L}(A)$, and
\begin{itemize}
\item[$\mathcal{L}_+(A)$], which denotes the collection of all vertical and horizontal lines passing through a point in $A$.
\end{itemize}
We again need each of their cardinality to be strictly less than $\mathfrak{c}$ whenever $|A|<\mathfrak{c}$. $|\mathcal{L}(A)|<\mathfrak{c}$ and $|\mathcal{L}_+(A)|<\mathfrak{c}$ are straightforward.\footnote{Recall \cref{FNLA<c} and note that $|\mathcal{L}_+(A)|\leq2\otimes|A|=|A|<\mathfrak{c}$. Of course, we omit the dull case of $A$ being finite.} To see $|\lin_Q(A)|<\mathfrak{c}$, note that $\lin_Q(A)$ is obtained by closing $A$ under a countable family of operations.\footnote{In particular, $\langle\langle x_1,y_1\rangle,\langle x_2,y_2\rangle\rangle\mapsto\langle x_1+x_2,y_1+y_2\rangle$ and $\langle x,y\rangle\mapsto\langle qx,qy\rangle$ for every $q\in\mathbb{Q}$.} By \cite[Lemma 6.6]{MR1475462}, we have $|\lin_Q(A)|=|A|+\omega$. Clearly, $|\lin_Q(A)|<\mathfrak{c}$ whenever $|A|<\mathfrak{c}$.

\begin{theorem}\label{TMHFHF}
There is a pair of bijective Hamel functions $f,g\colon\mathbb{R}\to\mathbb{R}$ such that $f\cup g$ is a Mazurkiewicz set.
\end{theorem}
\begin{proof}
Let $\{\textbf{p}_\xi\}_{\xi<\mathfrak{c}}$ and $\{\ell_\xi\}_{\xi<\mathfrak{c}}$ be injective enumerations of $\mathbb{R}^2$ and all straight lines of $\mathbb{R}^2$ respectively. By transfinite induction on $\xi<\mathfrak{c}$, We construct a sequences of pairs of mutually disjoint partial functions from $\mathbb{R}$ into $\mathbb{R}$ whose unions can be identified as a pair of Hamel functions. Fix any $\xi<\mathfrak{c}$ and suppose $\{f_\zeta,g_\zeta\}_{\zeta<\xi}$ is a sequence of such partial functions. We want to define the next pair $\{f_\xi,g_\xi\}$ such that
\begin{itemize}
\item[(I)] $|f_\xi\cup g_\xi|\leq6$.
\item[(P\textsubscript{I})] $|\ell\cap\bigcup_{\zeta\leq\xi}f_\zeta|\leq1$ and $|\ell\cap\bigcup_{\zeta\leq\xi}g_\zeta|\leq1$ whenever $\ell$ is vertical or horizontal.
\item[(P\textsubscript{M})] $|\ell\cap\bigcup_{\zeta\leq\xi}f_\zeta\cup g_\zeta|\leq2$ for every straight line $\ell\subseteq\mathbb{R}^2$.
\item[(P\textsubscript{H})] Both $\bigcup_{\zeta\leq\xi}f_\zeta$ and $\bigcup_{\zeta\leq\xi}g_\zeta$ are linearly independent over $\mathbb{Q}$.
\item[(D\textsubscript{M})] $|\ell_\xi\cap\bigcup_{\zeta\leq\xi}f_\zeta\cup g_\zeta|=2$.
\item[(D\textsubscript{H})] $\textbf{p}_\xi\in\lin_Q(\bigcup_{\zeta\leq\xi}f_\zeta)\cap\lin_Q(\bigcup_{\zeta\leq\xi}g_\zeta)$.
\end{itemize}
It is simply to see that the constructed $f\coloneqq\bigcup_{\xi<\mathfrak{c}}f_\xi$ and $g\coloneqq\bigcup_{\xi<\mathfrak{c}}g_\xi$ are both bijections on $\mathbb{R}$ by (P\textsubscript{I}) and (D\textsubscript{M}). They are also Hamel bases of $\mathrm{R}^2$ by (P\textsubscript{H}) and (D\textsubscript{H}). Most importantly, $f\cup g$ is a Mazurkiewicz set by (P\textsubscript{M}) and (D\textsubscript{M}). We just need to show that the construction of $\{f_\xi,g_\xi\}$ is possible.

We first ensure that (D\textsubscript{H}) happens by choosing appropriate $f_\xi'$ and $g_\xi'$ while preserving (P\textsubscript{I}), (P\textsubscript{M}), and (P\textsubscript{H}). If $\textbf{p}_\xi\in\lin_Q(\bigcup_{\zeta<\xi}f_\zeta)$, then simply let $f_\xi'\coloneqq\emptyset$. Otherwise, we consider a neither vertical nor horizontal straight line $\ell_\xi^*$ passing through $\textbf{p}_\xi$ such that $\ell_\xi^*\cap\bigcup_{\zeta<\xi}f_\zeta\cup g_\zeta=\emptyset$ and let
\begin{center}
$f_\xi'\subseteq\ell_\xi^*\setminus\Big(
\underbrace{\textstyle\bigcup\mathcal{L}_+(\bigcup_{\zeta<\xi}f_\zeta)}_{(P_I)}
\cup
\underbrace{\textstyle\bigcup\mathcal{L}(\bigcup_{\zeta<\xi}f_\zeta\cup g_\zeta)}_{(P_M)}
\cup
\underbrace{\textstyle\lin_Q(\bigcup_{\zeta<\xi}f_\zeta)}_{(P_H)}
\Big)$
\end{center}
be a linearly independent set consisting of two points having equal distance to $\textbf{p}_\xi$ so that we have $\textbf{p}_\xi\in\lin_Q(f_\xi')$ guaranteed. Note that the choice is possible since $|\mathcal{L}_+(\bigcup_{\zeta<\xi}f_\zeta)|<\mathfrak{c}$, $|\mathcal{L}(\bigcup_{\zeta<\xi}f_\zeta\cup g_\zeta)|<\mathfrak{c}$, and $|\lin_Q(\bigcup_{\zeta<\xi}f_\zeta)|<\mathfrak{c}$.\footnote{Recall that $|\mathcal{L}_+(A)|<\mathfrak{c}$, $|\mathcal{L}(A)|<\mathfrak{c}$, and $|\lin_Q(A)|<\mathfrak{c}$ whenever $|A|<\mathfrak{c}$. This is indeed the case since $|\bigcup_{\zeta<\xi}f_\zeta|\leq|\bigcup_{\zeta<\xi}f_\zeta\cup g_\zeta|\leq6\otimes\xi=\xi<\mathfrak{c}$.} $g_\xi'$ shall be determined in the same way.\footnote{Note that we need minor modifications since $f_\xi^*$ has been chosen. In particular, the line passing through $\textbf{p}_\xi$ should avoid intersecting $f_\xi^*\cup\bigcup_{\zeta<\xi}f_\zeta\cup g_\zeta$, and the choice of $g_\xi^*$ should avoid $\bigcup\mathcal{L}_+(\bigcup_{\zeta<\xi}g_\zeta)\cup\bigcup\mathcal{L}(f_\xi^*\cup\bigcup_{\zeta<\xi}f_\zeta\cup g_\zeta)\cup\lin_Q(\bigcup_{\zeta<\xi}g_\zeta)$ whenever $\textbf{p}_\xi\not\in\lin_Q(\bigcup_{\zeta<\xi}g_\zeta)$.}

Secondly, we want to ensure that (D\textsubscript{M}) happens by choosing appropriate $f_\xi''$ and $g_\xi''$ while preserving again (P\textsubscript{I}), (P\textsubscript{M}), and (P\textsubscript{H}). We simply let $f_\xi''$ be a subset of
\begin{center}
$\ell_\xi\setminus\Big(
\underbrace{\textstyle\bigcup\mathcal{L}_+(f_\xi'\cup\bigcup_{\zeta<\xi}f_\zeta)}_{(P_I)}
\cup
\underbrace{\textstyle\bigcup\mathcal{L}(f_\xi'\cup g_\xi'\cup\bigcup_{\zeta<\xi}f_\zeta\cup g_\zeta)}_{(P_M)}
\cup
\underbrace{\textstyle\lin_Q(f_\xi'\cup\bigcup_{\zeta<\xi}f_\zeta)}_{(P_H)}
\Big)$
\end{center}
such that $\ell_\xi\cap(f_\xi''\cup f_\xi'\cup g_\xi'\cup\bigcup_{\zeta<\xi}f_\zeta\cup g_\zeta)$ is a singleton if $\ell_\xi$ is vertical or horizontal, or it is a linearly independent doubleton if $\ell_\xi$ is otherwise. Then $g_\xi''$ shall be determined in the same way such that $\ell_\xi\cap(g_\xi''\cup f_\xi''\cup f_\xi'\cup g_\xi'\cup\bigcup_{\zeta<\xi}f_\zeta\cup g_\zeta)$ is again a singleton if $\ell_\xi$ is vertical or horizontal, or it is empty if $\ell_\xi$ is otherwise.

Evidently, letting $f_\xi\coloneqq f_\xi'\cup f_\xi''$ and $g_\xi\coloneqq g_\xi'\cup g_\xi''$ is as needed.
\end{proof}

Recall that any rotation of a Hamel basis about the origin is still a Hamel basis. Immediately from \cref{TMHFHF}, we end our discussion with the next remark.

\begin{remark}
There is a Mazurkiewicz sets $M\subseteq\mathbb{R}^2$ that is a union of two Hamel functions, and so is each of its counterclockwise rotations by $\frac{\pi}{2}$, $\pi$, $\frac{3\pi}{2}$ about the origin.
\end{remark}

\end{document}